\documentclass[11pt]{amsart}

\usepackage{amsthm,amsmath,amssymb}
\usepackage{graphicx}
\usepackage{float}
\usepackage[colorlinks=true,
linkcolor=blue,
urlcolor=blue,
citecolor=blue]{hyperref}
\usepackage{mathtools}
\usepackage{relsize}
\usepackage{ytableau}
\usepackage{tikz}
 \usepackage{enumerate}
\usepackage[toc,page]{appendix}
\usepackage{lscape}
\usepackage{multirow}
\usepackage{adjustbox,expl3,etoolbox} 
\usepackage{longtable}
\usepackage{xcolor,afterpage}
\usepackage[margin = 3.1cm,top = 3cm,bottom = 2.5cm]{geometry}

\newtheorem{thm}{Theorem}[section]
\newtheorem{lem}[thm]{Lemma}
\theoremstyle{definition}

\newtheorem{conj}[thm]{Conjecture}

\newtheorem{exa}[thm]{Example}
\newtheorem{rmk}[thm]{Remark}
\newtheorem{prob}[thm]{Problem}

\newtheorem{quest}[thm]{Question}

\DeclareMathOperator\sym{Sym}

\DeclareMathOperator\tr{tr}

\letcs\replicate{prg_replicate:nn} \newcommand*\longsum[1][1]{%
	\mathop{\textnormal{%
			\clipbox{0pt 0pt {.5\width} 0pt}{$\displaystyle\sum$}%
			\replicate{#1}{\clipbox{{.5\width} 0pt {.4\width} 0pt}{$\displaystyle\sum$}}%
			\clipbox{{.6\width} 0pt 0pt 0pt}{$\displaystyle\sum$}}}%
		}

\begin{document}	

\title[]{$3$-setwise intersecting families of the Symmetric group}

{\author[Behajaina]{Angelot Behajaina\textsuperscript{1}}
	\thanks{\textsuperscript{1}Normandie Univ., UNICAEN, CNRS, Laboratoire de Mathématiques Nicolas Oresme, 14000 Caen, France}
	\email{angelot.behajaina@unicaen.fr}}
{\author[Maleki]{Roghayeh Maleki\textsuperscript{2}}
	\thanks{\textsuperscript{2} Department of Mathematics and Statistics, University of Regina, Regina, Saskatchewan S4S 0A2, Canada}
	\email{rmaleki@uregina.ca}}
{\author[Rasoamanana]{Aina Toky Rasoamanana\textsuperscript{3,4,5,6}}
	\thanks{\textsuperscript{3}Institut Polytechnique de Paris, 91764 Palaiseau, France}
				  \thanks{\textsuperscript{4} R3S-SAMOVAR - R\'eseaux, Syst\`emes, Services, S\'ecurit\'e}
			  	  \thanks{\textsuperscript{5} RST - D\'epartement R\'eseaux et Services de T\'el\'ecommunications}
			  	  \thanks{\textsuperscript{6} CNRS - Centre National de la Recherche Scientifique}
	\email{aina.rasoamanana@telecom-sudparis.eu}}
{\author[Razafimahatratra]{A. Sarobidy Razafimahatratra\textsuperscript{2,*}}
	\thanks{\textsuperscript{*} Corresponding author.}
	\email{sarobidy@phystech.edu}}

\begin{abstract}
	Given two positive integers $n\geq 3$ and $t\leq n$, the permutations $\sigma,\pi \in \sym(n)$ are $t$-setwise intersecting if they agree (setwise) on a $t$-subset of $\{1,2,\ldots,n\}$. A family $\mathcal{F} \subset \sym(n)$ is $t$-setwise intersecting if any two permutations of $\mathcal{F}$ are $t$-setwise intersecting. Ellis [Journal of Combinatorial Theory, Series A, 119(4), 825--849, 2012] conjectured that if $t\leq n$ and $\mathcal{F} \subset \sym(n)$ is a $t$-setwise intersecting family, then $|\mathcal{F}|\leq t!(n-t)!$ and equality holds only if $\mathcal{F}$ is a coset of a setwise stabilizer of a $t$-subset of $\{1,2,\ldots,n\}$.
	
	In this paper, we prove that if $n\geq 11$ and $\mathcal{F}\subset \sym(n)$ is $3$-setwise intersecting, then $|\mathcal{F}|\leq 6(n-3)!$. Moreover, we prove that the characteristic vector of a $3$-setwise intersecting family of maximum size lies in the sum of the eigenspaces induced by the permutation module of $\sym(n)$ acting on the $3$-subsets of $\{1,2,\ldots,n\}$.
\end{abstract}



\date{\today}

\maketitle

\section{Introduction}

The celebrated Erd\H{o}s-Ko-Rado (EKR) Theorem is one of the most important results in extremal combinatorics.
\begin{thm}[Erd\H{o}s-Ko-Rado, \cite{erdos1961intersection}]
	Let $n$ and $k$ be positive integers such that $n\geq 2k$. If $\mathcal{F}$ is a family of $k$-subsets of $[n]:= \{1,2\ldots,n\}$ with the property that $A\cap B \neq \varnothing$ for all $A,B\in \mathcal{F}$, then $|\mathcal{F}|\leq \binom{n-1}{k-1}$. If $n>2k$, then $|\mathcal{F}| = \binom{n-1}{k-1}$ if and only if $\mathcal{F}$ consists of all the $k$-subsets of $[n]$ containing a fixed element.\label{thm:EKR}
\end{thm}
Theorem~\ref{thm:EKR} has been thoroughly studied and extended for various combinatorial objects in the past 50 years \cite{cameron2003intersecting,ellis2012triangle,Frankl1977maximum,frankl1986erdos}. For instance, Deza and Frankl \cite{Frankl1977maximum} extended the EKR theorem for the symmetric group. An extension for permutation groups, in general, was considered later in \cite{ahmadi2014new,meagher2016erdHos,wang2008erdHos}. Given a finite permutation group $G\leq \sym(\Omega)$, we say that two permutations $g$ and $h$ are \emph{intersecting} if they agree on some $\omega \in \Omega$, that is, $\omega^g = \omega^h$.
A family $\mathcal{F}$ $\subset G$ is intersecting  if any two elements of $\mathcal{F}$ are intersecting. We say that a transitive group $G\leq \sym(\Omega)$ has the \emph{Erd\H{o}s-Ko-Rado property} or \emph{EKR property} if any intersecting family $\mathcal{F} \subset G$ is of size at most $\frac{|G|}{|\Omega|}$. Moreover, $G$ has the \emph{strict-EKR property} if the only intersecting families of size $\frac{|G|}{|\Omega|}$ are cosets of a stabilizer of a point. Many families of groups have been proved to have the EKR property. For instance, the outstanding result by Meagher, Spiga, and Tiep \cite{meagher2016erdHos} shows that every finite $2$-transitive groups have the EKR property. In 1977, Deza and Frankl also proved that the symmetric group with its natural action on $[n]$ has the EKR property \cite{Frankl1977maximum}. The characterization of the maximum intersecting families of the symmetric group turned out to be harder, as it was only proved in 2004 independently by Cameron and Ku \cite{cameron2003intersecting}; and Larose and Malvenuto \cite{larose2004stable}.

Let $G$ be a finite group and let $C$ be a subset of $G\setminus \{e\}$ which is inverse-closed (i.e., $x^{-1}\in C$ whenever $x\in C$). Recall that the Cayley graph $\operatorname{Cay}(G,C)$ of the group $G$ with connection set $C$ is the graph whose vertex-set is $G$, and two group elements $g$ and $h$ are adjacent if and only if $gh^{-1} \in C$. We also recall that a Cayley graph is vertex transitive and regular of degree equal to the size of its connection set. 

A typical way of solving an EKR-type problem is by a reduction to a  graph invariant. Given a finite permutation group $G\leq \sym(\Omega)$, the \emph{derangement graph} of $G$ is 
the Cayley graph $\Gamma_G = \operatorname{Cay}(G,D)$ with connection set $D = \left\{ g\in G : \omega^g \neq \omega,\mbox{ for any }  \omega \in \Omega \right\}$. The derangement graph $\Gamma_G$ is a normal Cayley graph since its connection set, $D$, is a union of conjugacy classes.
By definition of the edges of $\Gamma_G$, $\mathcal{F}$ is an independent set or coclique of $\Gamma_G$ if and only if $\mathcal{F}$ is an intersecting family of $G$.

For $n\geq t\geq 2$, we say that two permutations $\sigma$ and $\pi$ of $\sym(n)$ are $t$\emph{-setwise intersecting }if there exists a $t$-subset $S$ of $[n]$ such that $S^\sigma = S^\pi$. The family $\mathcal{F} \subset \sym(n)$ is $t$\emph{-setwise intersecting} if any two permutations of $\mathcal{F}$ are $t$-setwise intersecting. These terminologies were introduced in \cite{ellis2012setwise,jackson2013properties}. 
By the reduction presented above, a $t$-setwise intersecting family of $\sym(n)$ is a coclique in a certain Cayley graph of $\sym(n)$. This Cayley graph is the $t$\emph{-derangement graph} $\Gamma_{n,t}$ of $\sym(n)$. We say that a permutation of $\sym(n)$ is a $t$\emph{-derangement} if its cycle type does not contain a composition of the integer $t$. Equivalently, $\sigma \in \sym(n)$ is a $t$-derangement if $\sigma$ does not fix any $t$-subset of $[n]$. The $t$-derangement graph $\Gamma_{n,t}$ is the graph $\operatorname{Cay}(\sym(n),D_{n,t})$, where $$D_{n,t} := \left\{ \sigma \in \sym(n) :  \sigma \mbox{ is a $t$-derangement}\right\}.$$

The following was conjectured by Ellis \cite{ellis2012setwise}.
\begin{conj}
	Let $n$ and $t$ be positive integers such that $n\geq t$. If $\mathcal{F}$ is a $t$-setwise intersecting family of $\sym(n)$, then $|\mathcal{F}| \leq t!(n-t)!$. Moreover, if $(n,t)\not \in \{(4,2),(5,2)\}$ and $|\mathcal{F}| = t!(n-t)!$, then $\mathcal{F}$ is a coset of a setwise stabilizer of a $t$-subset of $\{1,2,\ldots,n\}$.\label{conj:setwise}
\end{conj}
Ellis \cite{ellis2012setwise} proved that Conjecture~\ref{conj:setwise} holds asymptotically. That is, for a fixed $t$, there exists $n_0(t) \in \mathbb{N}$ such that Conjecture~\ref{conj:setwise} holds when $n\geq n_0$. Meagher et al.  \cite{meagher20202} recently proved that for $n \geq 2$, the upper bound on $t$-setwise intersecting families in Conjecture~\ref{conj:setwise} holds when $t=2$.  These results were proved using the representation theory of the symmetric group and the eigenvalue method. The eigenvalue method was first used by Wilson \cite{wilson1984exact} to prove the exact bound for the Erd\H{o}s-Ko-Rado theorem. The same method was also used by Godsil and Meagher in \cite{godsil2009new,godsil2016algebraic}, and by Ellis, Friedgut, and Pilpel  in \cite{ellis2011intersecting} to prove EKR-type results on the symmetric group.

In this paper, we prove a similar result to \cite{meagher20202}. We show that the bound on the size of the largest family in Conjecture~\ref{conj:setwise} holds when $t=3$. 
\begin{thm}
	Let $n\geq 11$. If $\mathcal{F}$ is a $3$-setwise intersecting family of permutations of $\sym(n)$, then $|\mathcal{F}| \leq 6(n-3)!$. Moreover, if $|\mathcal{F}| = 6(n-3)!$, then the characteristic vector $v_\mathcal{F}$ of $\mathcal{F}$ is in the direct sum of the eigenspaces determined by the constituents of the permutation character of $\operatorname{Sym}(n)$ acting on the $3$-subsets of $\{1,2,\ldots,n\}$.\label{thm:main} 
\end{thm}
Our proof uses a similar approach to that of \cite{meagher20202}. In particular, we use the representation theory of the symmetric group and the eigenvalue method on a weighted adjacency matrix of the $3$-derangement graph. Our proof is more challenging due to the fact that the conjugacy classes of $3$-derangements that we use are not just long cycles. For instance, we use conjugacy classes that have cycles of length $n-5$ and $n-6$. In contrast to the proof in \cite{meagher20202} where the values of the irreducible characters on the selected conjugacy classes of $2$-derangements are in $\{0,\pm 1\}$, some irreducible characters of $\sym(n)$ have larger values (in absolute value) on some of the conjugacy classes of $3$-derangements containing $n-6$ in their cycle type.

This paper is organized as follows. In Section~\ref{sect:background}, we give some background material used in the proof of the main result. In Section~\ref{sect:representation}, we recall some preliminary results on the representation of the symmetric group. The proof of Theorem~\ref{thm:main} occupies the next three sections. In Section~\ref{sect:proof-small}, we give a sketch of our proof to Theorem~\ref{thm:main}, and we prove that it holds for $11\leq n\leq 19$ and for $n \in \{21,23,25\}$. The remaining part of the proof of Theorem~\ref{thm:main} is divided according to the parity of $n$. The proof of Theorem~\ref{thm:main}, when $n\geq 27$ is odd, is given in Section~\ref{sect:odd}. In Section~\ref{sect:proof-even}, we give the proof of Theorem~\ref{thm:main}, when $n\geq 20$ is even. We end this paper by giving some interesting questions in Section~\ref{sect:further_works}.

\subsection*{Notations and definitions} 
All groups considered in this paper will be finite. The following notations will be used throughout this work.
Given a group $G$, we denote by $\operatorname{IRR}(G)$ a complete set of distinct irreducible $\mathbb{C}$-representations of $G$. For any $\mathfrak{X} \in \operatorname{IRR}(G)$, we denote by $\chi _{\mathfrak{X}}$ the character afforded by $\mathfrak{X}$. That is, $\chi_\mathfrak{X}(g) = \tr (\mathfrak{X}(g))$, for any $g\in G$. We denote by $\operatorname{Irr}(G)$ the set of all irreducible characters of $G$; that is, $\operatorname{Irr}(G) := \left\{ \chi_{\mathfrak{X}} : \mathfrak{X}\in \operatorname{IRR}(G) \right\}$. When $G = \sym(n)$, we let $\operatorname{Irr}_n := \operatorname{Irr}(G)$ and $\operatorname{IRR}_n := \operatorname{IRR}(G)$.

We say that the sequence of positive integers $\lambda = [\lambda_1,\lambda_2,\ldots,\lambda_k]$ is a \emph{partition} of $n$, and denote by $\lambda \vdash n$, if $\sum_{i=1}^k \lambda_i = n$ and $\lambda_1 \geq \lambda_2 \geq \ldots \geq \lambda_k\geq 1$. We say that the sequence of positive integers $\lambda = (\lambda_1,\lambda_2,\ldots,\lambda_k)$ is a \emph{composition} of $n$ if $\sum_{i=1}^k \lambda_i =n$. 
Let $\sigma \in \sym(n)$ and $\lambda = (\lambda_1,\lambda_2,\ldots,\lambda_k)$ be its cycle type. 
The conjugacy class of $\sym(n)$ containing $\sigma$ will be denoted by $C_{\lambda} = C_{(\lambda_1,\lambda_2,\ldots,\lambda_k)}$.

\section{Background}\label{sect:background}

In this section, we present some results that are used in our proof. We will recall some important bounds on the independence number of vertex-transitive graphs, namely the clique-coclique bound and the ratio bound. Then, we give a method to compute the eigenvalues of a normal Cayley graph from the irreducible characters of its group.
\subsection{Upper bounds on the independence number}
Since the first upper bound that we will present is in function of the clique number, we will first revisit some classic upper bounds on the clique number and the chromatic number.

Let $G\leq \sym(n)$ be a transitive permutation group and let $\Gamma_G$ be its derangement graph. Let $C$ be a clique of $\Gamma_G$ of maximum size; that is, of size $\omega(\Gamma_G)$. Without loss of generality, we suppose that $C$ contains the identity element $id$. As every non-identity element of $C$ does not intersect with $id$, every non-identity element of $C$ is a derangement. So the identity element is the only element of $C$ fixing $1$. Since $C$ is a clique, there is at most one permutation of $C$ mapping $1$ to $2$. Similarly, there is at most one permutation of $C$ that maps $1$ to $i$, for any $i\in \{3,4,\ldots,n\}$. We conclude that
\begin{align}
	\omega(\Gamma_G) \leq n.\label{eq:largest_clique}
\end{align}

The chromatic number of $\Gamma_G$ behaves in the same manner. Let $G_{i\to j}$ be the set of all permutations of $G$, which map $i$ to $j$, for $i,j\in [n]$. It is straightforward that $G_{i\to j}$ is a coclique of $\Gamma_G$ for any $i,j\in [n]$, and $G_{1\to 1}$ is the stabilizer of $1$ in $G$ . The $G_{1\to 1}$-cosets of $G$ are $G_{1\to 1},\ G_{1\to 2},\ \ldots,G_{1\to n}$. Therefore, assigning distinct colors to these $G_{1\to 1}$-cosets yields a proper coloring of $\Gamma_G$. In other words,
\begin{align}
	\chi(\Gamma_G) \leq n.
\end{align}

The first bound on the independence  number that we present is the \emph{clique-coclique} bound.
\begin{lem}[{Clique-coclique} \cite{godsil2016erdos}]
	Let $X = (V,E)$ be a vertex-transitive graph. Then $\alpha(X)\omega(X) \leq |V(X)|$. Equality holds if and only if each clique intersects each coclique at a unique vertex.\label{lem:clique-coclique}
\end{lem}
\begin{exa}
	We use the clique-coclique bound to prove that the dihedral group $D_{2n}$ acting on the regular $n$-gon has the EKR property. Let $\Gamma$ be the derangement graph of $D_{2n}$. Since $D_{2n}$ contains an $n$-cycle $\sigma$, the subgroup $C = \langle \sigma \rangle$ is a transitive subgroup of $D_{2n}$ of order $n$ and so every non-identity element of $C$ is a derangement. In particular, $C$ is a regular subgroup of $D_{2n}$, meaning that $C$ is a clique of size $n$. By \eqref{eq:largest_clique}, we have $\omega(\Gamma) = n$. By Lemma~\ref{lem:clique-coclique}, we have $\alpha(\Gamma) \leq \frac{|D_{2n}|}{\omega(\Gamma)} \leq \frac{2n}{n} = 2$, which is equal to the size of a stabilizer of a point of $D_{2n}$. Hence, $\alpha(\Gamma) = 2.$
\end{exa}

Next, we will see an upper bound on the size of a coclique in function of the eigenvalues of the graph (i.e., eigenvalues of the adjacency matrix of the graph). This upper bound is known as the Hoffman bound, the Delsarte bound, or the ratio bound \cite{delsarte1973algebraic,godsil2016erdos}. In this paper, we will refer to it as the ratio bound.

Let $G\leq \sym(n)$. The \emph{characteristic vector} $v_S$ of $S\subset G$ is the $\{0,1\}$-vector indexed by elements of $G$ such that the entry $v_S(g)$ is $1$ if $g\in S$ and $0$ otherwise. 

\begin{lem}[Ratio bound \cite{delsarte1973algebraic,godsil2016erdos,haemers2021hoffman}]
	Let $X=(V,E)$ be a $d$-regular graph and let $\tau$ be the smallest eigenvalue of $X$. Then 
	\begin{align*}
	\alpha(X) \leq |V(X)| \left( 1- \frac{d}{\tau} \right)^{-1}.
	\end{align*}
	If $S$ is a maximum coclique attaining the upper bound and $v_S$ is the characteristic vector of $S$, then $v_S - \frac{|S|}{|V|}\mathbf{1}$ is a $\tau$-eigenvector of the adjacency matrix of $X$, where $\mathbf{1}$ is the all one vector of $\mathbb{R}^{|V(X)|}$.
\end{lem}
The ratio bound can also be generalized for graphs related to weighted adjacency matrices. A \emph{weighted adjacency matrix} $A$ corresponding to a graph $X=(V,E)$ is a real symmetric matrix with constant row and column sum, which is related to the graph $X=(V,E)$ in the sense that $A(i,j) = 0$ if $i\not\sim_X j$.
\begin{lem}[Weighted ratio bound \cite{godsil2016erdos}]
	Let $X = (V,E)$ be a graph, and $A$ be a weighted adjacency matrix of $X$. Suppose that $\tau$ is the minimum eigenvalue of $A$ and that the row sum of $A$ is equal to $d$. Then,
	\begin{align*}
	\alpha(X) \leq |V(X)| \left( 1- \frac{d}{\tau} \right)^{-1}.
	\end{align*}
		If $S$ is a maximum coclique attaining the upper bound, then $v_s - \frac{|S|}{|V|}\mathbf{1}$ is a $\tau$-eigenvector of $A$.\label{lem:ratio-bound}
\end{lem}

\subsection{Eigenvalues of normal Cayley graphs and association schemes}
 
The following result gives the eigenvalues of normal Cayley graphs, i.e., Cayley graphs whose connection set is invariant under conjugation. 
\begin{thm}[\cite{babai1979spectra}]
	Let $X = \operatorname{Cay}(G,C)$ be a Cayley graph such that $C$ is invariant under conjugation. Let $ (\mathfrak{X}_1,V_1),\ (\mathfrak{X}_2, V_2),\ldots,\ (\mathfrak{X}_k,V_k )$ be a complete list of distinct irreducible representations of $G$ and let $\chi_i$ be the character afforded by $\mathfrak{X}_i$, for any $i\in \{1,2,\ldots,k\}$. Then,
	\begin{align*}
		\mathbb{C}G \cong \bigoplus_{i=1}^k  U_i,
	\end{align*}
	where $U_i$ is the sum of all submodules  isomorphic to $V_i$ in the regular $\mathbb{C}G$-module.
	Moreover, $U_i$ is an eigenspace of $X$ with eigenvalue 
	\begin{align*}
		\xi_{\mathfrak{X}_i} = \frac{1}{\dim \mathfrak{X}_i} \sum_{g\in C} \chi_i(g),
	\end{align*}
	for any $i \in \{1,2,\ldots,k\}$.\label{lem:eig-norm}
\end{thm}

We give a generalization of Theorem~\ref{lem:eig-norm} for certain weighted adjacency matrices. These weighted adjacency matrices will involve certain spanning subgraphs of a normal Cayley graph. First, we define an association scheme induced by the group $G$, then we use this association scheme to write the eigenvalues of the weighted adjacency matrix.

Let $G$ be a group with distinct non-identity conjugacy classes $C_1,C_2,\ldots,C_k$; that is, $C_{i}\neq \{id\}$ for any $i\in \{1,2,\ldots,k\}$. The \emph{conjugacy class scheme} of $G$ is the association scheme given by the matrices $(A_i)_{i=1,\ldots,k}$, where $A_i$ is the $|G|\times |G|$ matrix whose rows and columns are indexed by the group elements and whose $(u,v)$-entry is
\begin{align*}
	A_i(u,v) =
	\left\{
		\begin{aligned}
			&1 \ \ \mbox{ if } uv^{-1} \in C_i,\\
			&0 \ \ \mbox{ otherwise.}
		\end{aligned}
	\right.
\end{align*}
The matrix $A_i$ is the adjacency matrix of the digraph $\operatorname{Cay}(G,C_i)$, for any $i\in \{1,2\ldots,k\}$. Note that when the conjugacy classes of $G$ are inverse-closed, then the matrices $(A_i)_{i=1,2,\ldots,k}$ are symmetric and $\operatorname{Cay}(G,C_i)$ is an undirected graph, for any $i\in \{1,2,\ldots,k\}$. Since the conjugacy classes of the symmetric group are inverse-closed, its conjugacy class scheme is a symmetric association scheme.

For the remainder of this subsection, we suppose that the conjugacy classes of $G$ are inverse-closed. It is easy to see that any real linear combination of the matrices $(A_i)_{i=1,\ldots,k}$ is a weighted adjacency matrix that corresponds to a certain spanning subgraph of the normal Cayley graph $\operatorname{Cay}(G, G\setminus \{id\})$.

\begin{lem}[\cite{godsil2016erdos}]
	Let $(\omega_i)_{i=1,\ldots,k} \subset \mathbb{R}$ and $G$ be a group with non-identity conjugacy classes $C_1,C_2,\ldots,C_k$. Let $A = \sum_{i=1}^k \omega_iA_i$, where $A_i$ is the matrix of the conjugacy class $C_i $ in the conjugacy class scheme of $G$, for any $i\in \{1,2,\ldots,k\}$. The eigenvalues of $A$ are of the form
	\begin{align*}
		\xi_{\chi} &= \sum_{i=1}^k \omega_i |C_i| \frac{\chi(g_i)}{\chi(id)},
	\end{align*}
	where $\chi \in \operatorname{Irr}(G)$, and $g_1,g_2,\ldots,g_k$ are, respectively, representatives of the conjugacy classes $C_1,C_2,\ldots, C_k$.\label{lem:evalue-wadj}
\end{lem}
\section{Representation theory of $\sym(n)$}\label{sect:representation}
In this section, we prove some results on the complex representations of the symmetric group that are needed in our proof of Theorem~\ref{thm:main}. For the basic definitions and notions on the representation theory of the symmetric group, we refer the reader to the textbook of Sagan \cite{sagan2001symmetric}. In particular, the reader is expected to be familiar with the Specht modules \cite[Section~2.3]{sagan2001symmetric}, the Branching Rule \cite[Section~2.8]{sagan2001symmetric}, the Hook Length Formula \cite[Section~3.10]{sagan2001symmetric} and the Murnaghan-Nakayama Rule \cite[Section~4.10]{sagan2001symmetric}. 

\subsection{Character values}
Let $\lambda \vdash n$. 
We denote the irreducible character corresponding to the $\lambda$-Specht module of $\sym(n)$ by $\chi^\lambda$.  
For any element $\sigma \in \sym(n)$ in the conjugacy class of cycle type $(p_1,p_2,\ldots,p_k)$, we define
\begin{align*}
	\chi_{(p_1,p_2,\ldots,p_k)}^\lambda := \chi^\lambda(\sigma).
\end{align*}
Let $\lambda = \left[ \lambda_1,\lambda_2,\ldots,\lambda_k  \right] \vdash n$ and $\mu = \left[\mu_1,\mu_2,\ldots,\mu_l\right]\vdash m$, where $m< n$. We say that $\lambda$ \emph{contains} $\mu$, and write $\mu \subset \lambda$ if $\mu_i \leq \lambda_i$, for all $i\in \{1,2,\ldots,l\}$. When $\mu \subset \lambda$, the corresponding \emph{skew diagram} $\lambda/\mu$ is the set of cells of $\lambda$ that are not in $\mu$. That is, 
\begin{align*}
	\lambda/ \mu = \left\{ c\in \lambda \mid c\not\in \mu \right\}.
\end{align*} 
A \emph{rim hook} $\zeta$ of a Young diagram $\lambda \vdash n$ is a skew diagram of $\lambda$ whose cells are on a path with upward and rightward steps. The \emph{leg length} $\ell\ell(\zeta)$ of the rim hook $\zeta$ of $\lambda$ is the number of rows it spans minus $1$. 

Given a partition $\lambda$ and a rim hook $\zeta$ of $\lambda$, we define $\lambda\setminus \zeta$ to be the set of cells of $\lambda$ that are not in $\zeta$. Since $\zeta$ is a skew diagram, $\lambda \setminus \zeta$ is a Young diagram.

Given a composition $\rho = (a,\rho_2,\ldots,\rho_k)$ of $n$, we define the operation $\rho \setminus a$ to be the composition of $n-a$ obtained by deletion of the first entry (which is $a$) of $\rho$.

The next lemma is the famous  recursive Murnaghan-Nakayama rule. We state it for the sake of completeness.
\begin{lem}[Murnaghan-Nakayama rule \cite{sagan2001symmetric}]
	 Let $\lambda \vdash n$ and $\rho = (\rho_1,\rho_2,\ldots,\rho_k)$ be a composition of $n$. Then,
	\begin{align*}
	\chi^\lambda_{\rho} &= \longsum_{\zeta \in {RH_{\rho_1}}(\lambda)} (-1)^{\ell\ell(\zeta)} \chi_{\rho \backslash \rho_1}^{\lambda\setminus \zeta},
	\end{align*}
	where ${ RH}_{\rho_1}(\lambda)$ is the set of all rim hooks of length $\rho_1$ of $\lambda$.\label{lem:MurnNakRule}
\end{lem}


\subsection{Character degrees}
In this subsection, we find the irreducible representations of $\sym(n)$ of relatively small degree. To do this, we use the Hook Length Formula \cite[Section~3.10]{sagan2001symmetric} and the Branching Rule \cite[Section~2.8]{sagan2001symmetric}.

We define a coordinate system on Young diagrams. The coordinate of a Young diagram is given as follows:
\begin{enumerate}[i.]
	\item the origin $O$, of coordinate $(1,1)$, is the cell on the top left corner of the Young diagram,
	\item a cell of the Young diagram has coordinate $(i,j)$ if there is a path with $j-1$ rightward steps and $i-1$ downward steps from $O$ to the cell.
\end{enumerate}
If  the cell of coordinate $(i,j)$ (also called $(i,j)$-cell) exists in a Young diagram $\lambda$, then we write $(i,j) \in \lambda$. A $(1,.)$-cell and a $(.,1)$-cell of $\lambda$ are cells that are situated in the first row and the first column of $\lambda$, respectively. 


The following result is about the irreducible characters of $\sym(n)$ of relatively small degree. Our proof is similar to a proof in \cite{godsil2016erdos}.
\begin{lem}\label{lem:irrecharless}
	Let $n \geq 27$. If $\phi$ is a character of $\mathrm{Sym}(n)$ of degree less than $5\binom{n}{3}$ and $\chi^{\lambda}$ an irreducible constituent of $\phi$, then $\lambda$ is one of the following: $[n],[1^n],[n-1,1],[2,1^{n-2}],[n-2,2],[2^2,1^{n-4}],[n-2,1^2],[3,1^{n-3}],[n-3,3],[2^3,1^{n-6}],[n-3,1^3],[4,1^{n-4}],[n-3,2,1]$ or $[3,2,1^{n-5}]$. 
\end{lem}

\begin{proof}
	Here, we follow the proof of \cite[Lemma~12.7.3]{godsil2016erdos} by using induction on $n$. Computation via \verb|Sagemath| \cite{sagemath} shows that the result holds for $n=27,28$. Now, we suppose that the lemma is true for $n$ and $n-1$. 
	
	Let $\phi$ be a character of $\sym(n+1)$ of degree less than $5\binom{n+1}{3}$. We may assume that $\phi$ is irreducible. Denote by $\phi\downharpoonleft _{[n]}$ (resp. $\phi\downharpoonleft _{[n-1]}$) the restriction of $\phi$ to $\mathrm{Sym}(n)$ (resp. $\mathrm{Sym}(n-1)$).
	
	First, assume that $\phi\downharpoonleft _{[n]}$ has a constituent which is one of the fourteen irreducible characters of $\mathrm{Sym}(n)$ of degree less than $5\binom{n}{3}$. From the branching rule (see \cite[ Theorem~2.8.3]{sagan2001symmetric}), $\phi$ is one of the irreducible characters given in the middle column of Table \ref{constcahr}. As the degree of any character in Table \ref{degreerepgre} is larger than $5\binom{n+1}{3}$ when $n \geq 27$, one can deduce that $\phi$ is one of the irreducible characters corresponding to the statement of the lemma.
	
	Next, assume that $\phi\downharpoonleft _{[n]}$ has none of its constituents listed among the fourteen irreducible characters of $\sym(n)$ of degree less than $5\binom{n}{3}$. Moreover, suppose that $\phi\downharpoonleft _{[n]}$ has at least two constituents. 
	Then, the degree of $\phi\downharpoonleft _{[n]}$ is at least $10\binom{n}{3}>5 \binom{n+1}{3}$ when $n \geq 27$, which is a contradiction.
	
	Finally, we may assume that $\phi\downharpoonleft _{[n]}$ is irreducible and is not one of the fourteen irreducible characters of degree less than $5\binom{n}{3}$ of $\sym(n)$. From the branching rule, the Young diagram corresponding to $\phi$ must be rectangular, that is, $\phi=\chi^{[a^b]}$ for some integers $a,b \geq 1$. Again, using the branching rule, we can deduce that the constituents of $\phi\downharpoonleft _{[n-1]}$ correspond to the two partitions $\lambda'=[a^{b-1},a-2]$ and $\lambda''=[a^{b-2},a-1,a-1]$. Neither $\lambda'$ nor $\lambda''$ is one of the fourteen irreducible characters of degree $5\binom{n-1}{3}$ of $\sym(n-1)$, for $n \geq 27$.
	Therefore, the degree of $\phi$ is at least $10\binom{n-1}{3} > 5\binom{n+1}{3}$ when $n \geq 27$, which is a contradiction.
	This completes the proof.
\end{proof}
\begin{rmk}
	The irreducible characters in Lemma~\ref{lem:irrecharless} all have degree less than $3\binom{n}{3}$. In fact, for $n\geq 19$, the irreducible constituents in Lemma~\ref{lem:irrecharless} are all the irreducible characters of degree less than $3\binom{n}{3}.$\label{rmk}
\end{rmk}

\begin{table} [t]
	\centering
	\begin{longtable}{| c | c | c |}
		\hline
		\textbf{Constituent of $\phi\downharpoonleft _{[n]}$} & \textbf{$\phi$} & \textbf{Degree of $\phi\downharpoonleft _{[n]}$}  \endhead
		\hline
		$[n]$&$[n+1],[n,1]$ &$1$\\
		\hline
		$[1^n]$&$[2,1^{n-1}],[1^{n+1}]$ &$1$\\
		\hline
		$[n-1,1]$&$[n,1],[n-1,2],[n-1,1^2]$ &$n-1$\\
		\hline
		$[2,1^{n-2}]$&$[2,1^{n-1}],[2^2,1^{n-3}],[3,1^{n-2}]$ &$n-1$\\
		\hline
		$[n-2,2]$&$[n-1,2],[n-2,3],[n-2,2,1]$ &$\frac{n(n-3)}{2}$\\
		\hline
		$[2^2,1^{n-4}]$&$[2^2,1^{n-3}],[2^3,1^{n-5}],[3,2,1^{n-4}]$&$\frac{n(n-3)}{2}$\\
		\hline
		$[n-2,1^2]$&$[n-1,1^2],[n-2,2,1],[n-2,1^3]$ & $\frac{(n-1)(n-2)}{2}$\\
		\hline
		$[3,1^{n-3}]$&$[3,1^{n-2}],[3,2,1^{n-4}],[4,1^{n-3}]$ & $\frac{(n-1)(n-2)}{2}$\\
		\hline
		$[n-3,3]$&$[n-2,3],[n-3,4],[n-3,3,1]$ & $\frac{n(n-1)(n-5)}{6}$\\
		\hline
		$[2^3,1^{n-6}]$&$[2^3,1^{n-5}],[2^4,1^{n-7}],[3,2^2,1^{n-6}]$ & $\frac{n(n-1)(n-5)}{6}$\\
		\hline
		$[n-3,1^3]$&$[n-2,1^3],[n-3,2,1^2],[n-3,1^4]$ & $\frac{(n-1)(n-2)(n-3)}{6}$\\
		\hline
		$[4,1^{n-4}]$&$[4,1^{n-3}],[4,2,1^{n-5}],[5,1^{n-4}]$ & $\frac{(n-1)(n-2)(n-3)}{6}$\\
		\hline
		$[n-3,2,1]$&$[n-2,2,1],[n-3,3,1],[n-3,2^2],[n-3,2,1^2]$& $\frac{n(n-2)(n-4)}{3}$\\
		\hline
		$[3,2,1^{n-5}]$ &$[4,2,1^{n-5}],[3^2,1^{n-5}],[3,2^2,1^{n-6}],[3,2,1^{n-4}]$ & $\frac{n(n-2)(n-4)}{3}$\\
		\hline
		\caption{\label{constcahr} Shape of $\phi$ if $\phi\downharpoonleft _{[n]}$ has a constituent of degree less than $5\binom{n}{3}$.}
	\end{longtable}
\end{table}
\newpage
\begin{table}[H] 
	\begin{longtable}{| c | c |}
		\hline
		\textbf{Representation} & \textbf{Degree}  \endhead
		\hline
		$[n-3,4]$ & $\frac{(n+1)n(n-1)(n-6)}{24}$ \\
		\hline
		$[n-3,3,1]$ & $\frac{(n+1)n(n-2)(n-5)}{8}$ \\
		\hline
		$[2^4,1^{n-7}]$ & $\frac{(n+1)n(n-1)(n-6)}{24}$ \\
		\hline
		$[3,2^2,1^{n-6}]$ & $\frac{(n+1)n(n-2)(n-5)}{8}$ \\
		\hline
		$[n-3,2,1^2]$ & $\frac{(n+1)(n-1)(n-2)(n-4)}{8}$ \\
		\hline
		$[n-3,1^4] $ & $\frac{n(n-1)(n-2)(n-3)}{24}$ \\
		\hline
		$[4,2,1^{n-5}]$ & $\frac{(n+1)(n-1)(n-2)(n-4)}{8}$ \\
		\hline
		$[5,1^{n-4}]$ & $\frac{n(n-1)(n-2)(n-3)}{24}$ \\
		\hline
		$[n-3,2^2]$ & $\frac{(n+1)n(n-3)(n-4)}{12}$ \\
		\hline
		$[3^2,1^{n-5}]$ & $\frac{(n+1)n(n-3)(n-4)}{12}$ \\
		\hline
		\caption{\label{degreerepgre} Degree of the representations from Table \ref{constcahr} that are larger than $5\binom{n+1}{3}$.}
	\end{longtable}
\end{table} 
\subsection{Uniqueness of rim hooks}

We prove a result on the uniqueness of rim hooks that are relatively long.

\begin{lem}
	Suppose  $n\geq 1$ and $1\leq a \leq n$. Let $\lambda = [\lambda_1,\lambda_2,\ldots,\lambda_q]\vdash n$ and $\mu = [\mu_1,\mu_2,\ldots,\mu_t] \vdash a$ be a Young diagram obtained by removing a rim hook of length $n-a$ of $\lambda$.
	
	\begin{enumerate}[(1)]
		\item If $\tau$ is a rim hook of $\lambda$ that does not meet either $(1,.)$-cells or $(.,1)$-cells, then the length of $\tau$ is at most $a$.
		\item If $\tau$ is a rim hook of length $\ell \geq a+1$ of $\lambda$, then $\tau$ is either the unique rim hook of length $\ell$ starting at $(q,1)$ or ending at $(1,\lambda_1)$.
		\item If $\ell$ is a positive integer such that $2\ell - n \geq a+1$, then there is exactly one rim hook of length $\ell$ in $\lambda$.
		In particular, if $3a+1\leq n$, then there is a unique rim hook of length $n-a$ in $\lambda$.
	\end{enumerate}\label{lem:uniqueness}
\end{lem}
\begin{proof}
	(1) Since no cell of $\tau$ is in the first row or the first column of $\lambda$, the rim hook $\tau$ lies in the skew diagram $\lambda / [\lambda_1,1^{q-1}]$. The latter has at most $t$ rows and $\mu_1$ columns. Therefore, the length of $\tau$ is at most $\mu_1 +t-1 \leq a$.
	
	(2) Since the length of $\tau$ is $\ell \geq a+1$, by (1), $\tau$ must contain a cell of coordinate $(1,i_0)$ or $(j_0,1)$. If $\tau$ contains a $(1,i_0)$-cell, then every $(1,i)$-cell, for $i\geq i_0$, is a cell of $\tau$. That is, $\tau$ ends at the $(1,\lambda_1)$-cell. Since no cell can be below or on the right of $\tau$, it is impossible to have distinct rim hooks of length $\ell$ ending at the $(1,\lambda_1)$-cell. Similarly, if $\tau$ contains a $(j_0,1)$-cell, then $\tau$ must start at the $(q,1)$-cell. This also implies that $\tau$ is the unique rim hook of length $\ell$ starting at the $(q,1)$-cell.
	
	(3) Let $\gamma$ and $\gamma^\prime$ be two rim hooks of $\lambda$ that both have length $\ell$. If $\gamma$ and $\gamma^\prime$ do not have common cells, then $2\ell \leq n$, which contradicts the hypothesis. Hence $\gamma$ and $\gamma^\prime$ must have cells in common. By (1), the number of cells of $\gamma\cap \gamma^\prime$ in the skew diagram $\lambda/[\lambda_1,1^{q-1}]$ is at most $a$. However, by the fact that $2\ell - |\gamma \cap \gamma^\prime | = |\gamma| + |\gamma^\prime| - |\gamma \cap \gamma^\prime|\leq n$ and $2\ell -n \geq a+1$, we deduce that $|\gamma \cap \gamma^\prime| \geq a+1$. Hence, $\gamma\cap \gamma^\prime$ must contain a $(1,.)$-cell or a $(.,1)$-cell. It follows from (2) that $\gamma = \gamma^\prime$.
	The last result follows by taking $\ell =n-a$. 	
	This completes the proof.
\end{proof}

\section{The main idea of the proof of Theorem~\ref{thm:main} }\label{sect:proof-small}
The general idea for the proof of Theorem~\ref{thm:main} is to find a spanning subgraph $X$ of $\Gamma_{n,3}$ whose maximum coclique has size $6(n-3)!$. As removal of edges can only increase the size of the maximum cocliques, we will then have $\alpha(\Gamma_{n,3}) \leq \alpha(X) = 6(n-3)!$. 

One way of finding such a spanning subgraph is by assigning weights to the conjugacy classes of $3$-derangements of $\sym(n)$. This gives a spanning subgraph of $\Gamma_{n,3}$ which corresponds to a weighted adjacency matrix $A$.  
If the weighted adjacency matrix $A$ has all of its eigenvalues in $[-1,\binom{n}{3}-1]$ and its minimum and maximum eigenvalues are, respectively, $-1$ and $\binom{n}{3}-1$, then by applying Lemma~\ref{lem:ratio-bound}, we have
\begin{align*}
	\alpha(X) \leq n!\left(1-\frac{\binom{n}{3}-1}{-1}\right)^{-1} = 3!(n-3)!.
\end{align*}

In our proof, we find a weighted adjacency matrix $A$ so that
\begin{enumerate}[(i)]
	\item its maximum eigenvalue, $\binom{n}{3}-1$, is given by the irreducible character $\chi^{[n]}$, \label{max}
	\item its minimum eigenvalue, $-1$, is given by the  irreducible characters $\chi^{[n-1,1]},\ \chi^{[n-2,2]}$, and $\chi^{[n-3,3]}$, \label{min}
	\item any other irreducible character gives an eigenvalue which is in the interval $\left( -1,\binom{n}{3}-1 \right)$. \label{others}
\end{enumerate}

The choice of the irreducible characters $\chi^{[n]},\ \chi^{[n-1,1]},\ \chi^{[n-2,2]}$, and $\chi^{[n-3,3]}$ is related to the decomposition of the permutation character of the group $\sym(n)$ acting on the $3$-subsets of $\{1,2,\ldots,n\}$. Using the Young's rule \cite[Theorem~2.11.2]{sagan2001symmetric}, it is easy to derive that the irreducible constituents of the permutation character of $\sym(n)$ acting on the $3$-subsets of $\{1,2,\ldots,n\}$ are $\chi^{[n]},\ \chi^{[n-1,1]},\ \chi^{[n-2,2]}$, and $\chi^{[n-3,3]}$ (these were also computed in \cite[Lemma~15]{ellis2012setwise}). Therefore, our choice for $A$ is so that the irreducible characters producing the minimum eigenvalue $-1$ are exactly the non-trivial constituents of the permutation character.

For $11\leq n\leq 19$ and $n\in \{21,23,25\}$, we use \verb|Sagemath| \cite{sagemath} to find weightings on the conjugacy classes so that properties (\ref{max}), (\ref{min}), and (\ref{others}) are satisfied. In fact, there are infinitely many weighted adjacency matrices that we could find for these cases. The \verb|Sagemath| code that we used to find the weightings and compute the corresponding eigenvalues are available in \cite{Razafimahatratra2021}. 

The proof for the case $n\geq 27$ odd is given in Section~\ref{sect:odd}. For $n\geq 20$ even, we give the proof in Section~\ref{sect:proof-even}.

The proof of the second statement of Theorem~\ref{thm:main} follows from the ratio bound. If (\ref{max}), (\ref{min}), and (\ref{others}) are satisfied, then the ratio bound holds with equality. By Lemma~\ref{lem:ratio-bound}, the translated characteristic vector $v_\mathcal{F} - \frac{1}{\binom{n}{3}}\mathbf{1}$ of an intersecting family $\mathcal{F}$ of size $6(n-3)!$ is a $(-1)$-eigenvector of the weighted adjacency matrix $A$. Therefore, $v_\mathcal{F} - \frac{1}{\binom{n}{3}}\mathbf{1} \in U_{[n-1,1]} \oplus U_{[n-2,2]} \oplus U_{[n-3,3]}$, where $U_\lambda$ is the eigenspace that corresponds to the  $\lambda$-Specht module of $\sym(n)$ (see Theorem~\ref{lem:eig-norm}). Since the maximum eigenvalue of $A$ is its row sum, by (\ref{max}), we deduce that $\binom{n}{3}-1$ equals the row sum of $A$. Since the eigenspace $U_{[n]}$ is one dimensional, we conclude that $U_{[n]} = \langle \mathbf{1} \rangle$. 
\newpage
\noindent Hence,
\begin{align*}
	v_\mathcal{F} \in U_{[n]} \oplus U_{[n-1,1]} \oplus U_{[n-2,2]} \oplus U_{[n-3,3]}.
\end{align*}
\begin{rmk}
	The result in the first part of Theorem~\ref{thm:main} (i.e., the bound) also holds for $ 3\leq n \leq 10$. The \verb|Sagemath| code containing the necessary computations are also available in \cite{Razafimahatratra2021}.
\end{rmk}

\section{$3$-setwise intersecting when $n\geq 27$ is odd}\label{sect:odd}

Let $\mu$ be a composition of the integer $n$. We denote by $C_\mu$ the conjugacy class of $\sym(n)$ that has cycle type $\mu$, and $A_\mu$ the matrix of $C_\mu$ in the conjugacy class scheme of $\sym(n)$. We consider the weighted adjacency matrix
\begin{align}
	A &= x_1 A_{(n)} +x_2 A_{(n-2,1^2)} +x_3 A_{(n-2,2)} + x_4 A_{(n-5,4,1)} + x_5 A_{(n-1,1)}.\label{eq:odd-adj}
\end{align}
We find the weights $(x_i)_{i=1,2,3,4,5}$ in \eqref{eq:odd-adj}, so that $A$ verifies the properties \eqref{max}, \eqref{min}, and \eqref{others} described in Section~\ref{sect:proof-small}. The values of the irreducible characters $\chi^{[n]}$, $\chi^{[n-1,1]}, \ \chi^{[n-2,2]}$, and $\chi^{[n-3,3]}$ on an element of the conjugacy classes $C_{(n)},$   $C_{(n-2,1^2)},$  $C_{(n-2,2)},\ C_{(n-5,4,1)}$, and $C_{(n-1,1)}$ are given in the following table.
\begin{table}[h!]
	\begin{longtable}{|c|c|c|c|c|c|} \hline
		& $C_{(n)}$   & $C_{(n-2,1^2)}$ & $C_{(n-2,2)}$ & $C_{(n-5,4,1)}$ & $C_{(n-1,1)}$ \\
		Representation  & & & & &\\ \hline
		$\chi^{[n]}$      &  $1 $ &$1$ & $1$  & $1$ &  $1$  \\ \hline
		$\chi^{[n-1,1]}$ & $-1$ & $1$  & $-1$ & $0$ & $0$\\ \hline
		$\chi^{[n-2,2]}$ &$0$ &  $-1$ & $1$ & $-1$ & $-1$ \\ \hline
		$\chi^{[n-3,3]}$ &$0$ &  $-1$ & $-1$ & $0$ & $0$ \\ \hline
		\caption{Values of the constituents of the permutation character on the desired conjugacy classes.}
	\end{longtable}
\end{table}

\noindent Let $C_1:= C_{(n)},\ C_2 := C_{(n-2,1^2)} ,\ C_3 := C_{(n-2,2)},\ C_4 := C_{(n-5,4,1)}$, and $C_5 := C_{(n-1,1)}$. Let $\alpha= \binom{n}{3}-1$ and let $\beta,\gamma$, and $\delta$ be, respectively, the degree of the Specht module corresponding to $\ [n-1,1],\ [n-2,2]$, and $[n-3,3]$. That is, $\beta = (n-1)$, $\gamma = \binom{n}{2}-n$, and $\delta = \binom{n}{3} - \binom{n}{2}$. Finding the weights $(x_i)_{i=1,2,3,4,5}$ in \eqref{eq:odd-adj} for which $A$ satisfies \eqref{max} and \eqref{min} is equivalent to finding the solutions of the system of linear equations

\begin{align}
	\left\{
		\begin{aligned}
				&\omega_1 + \omega_2 + \omega_3 +\omega_4 + \omega_5 = \alpha\\ 
				& -\omega_1 + \omega_2 - \omega_3 = -\beta\\
				&-\omega_2 + \omega_3 -\omega_4 -\omega_5 = - \gamma \\
				& -\omega_2 - \omega_3 = -\delta,
		\end{aligned}
	\right.\label{eq:main-lin}
\end{align}
where $\omega_i = x_i |C_i|$ for $i\in \{1,2,3,4,5\}$.

The system of linear equations \eqref{eq:main-lin} has infinitely many solutions. A general solution to it is
\begin{align}
	\begin{split}
	\omega_1(s,t) &= -s- t + (\beta +\gamma)\\
	\omega_2(s,t) &= -\frac{1}{2}s-\frac{1}{2} t + \frac{1}{2} (\alpha -\beta)\\
	\omega_3(s,t) &= \frac{1}{2}s + \frac{1}{2}t + \frac{1}{2} (\alpha -\beta) -\gamma\\
	\omega_4(s,t) &= s\\
	\omega_5(s,t) &= t.
	\end{split}\label{eq:solution_odd}
\end{align}

Let $A$ be the weighted adjacency matrix defined in \eqref{eq:odd-adj} with the weights in \eqref{eq:solution_odd}. We prove that the eigenvalues of $A$ are in the interval $[-1,\binom{n}{3}-1]$, for appropriate values of $s$ and $t$. 
\subsection{Irreducible characters  on $C_1,C_2,C_3,C_4,$ and $C_5$}
In this subsection, we give the values of the irreducible characters of $\sym(n)$ on the conjugacy classes $C_1,\ C_2,\ C_3,\ C_4$, and $C_5$. We will see in particular that these values are in the set $\{-1,0,1\}$.
\begin{lem}
	If $\chi \in \operatorname{Irr}_n$, then $|\chi(x)| \in \{0,1\}$ for any $x\in C_1 \cup C_2\cup C_3\cup C_4 \cup C_5$.\label{lem:main_val}
\end{lem}
\begin{proof}
	The proof is an immediate consequence of Lemma~\ref{lem:uniqueness} and the Murnaghan-Nakayama rule. Let $\lambda \vdash n$. For $n\geq 16$, Lemma~\ref{lem:uniqueness} implies that there exists at most one rim hook of length $n-a$ in the Young diagram corresponding to $\lambda \vdash n$, for $a\in \{0,1,2,5\}$. It is easy to see that $\chi^\lambda_{(n)} \in \{-1,0,1\}$ as there exists at most one rim hook of length $n$. For the other conjugacy classes, by the Murnaghan-Nakayama rule, we have
	\begin{align*}
		|\chi^\lambda_{(n-1,1)}| & \leq \max_{\mu \vdash 1} |\chi^\mu_{(1)}| = 1,\\
		|\chi^\lambda_{(n-2,1^2)}| & \leq \max_{\mu \vdash 2} |\chi^\mu_{(1^2)}| = 1\\
		|\chi^\lambda_{(n-2,2)}| & \leq \max_{\mu \vdash 2} |\chi^\mu_{(2)}| = 1,\\
		|\chi^\lambda_{(n-5,4,1)}| & \leq \max_{\mu \vdash 5} | \chi^\mu_{(4,1)}| = 1.
	\end{align*}
	This completes the proof since the symmetric group has integral characters.  
\end{proof}

\subsection{Eigenvalues of small degree characters}\label{subs:small_char_odd}
In this subsection, we compute the eigenvalues belonging to characters of degree less than $5\binom{n}{3}$ in terms of $s$ and $t$ (see \eqref{eq:solution_odd}). Then, we define a polytope of $\mathbb{R}^2$ where these eigenvalues are in the interval $[-1,\binom{n}{3}-1]$, for every $(t,s)$ in the polytope.

By Lemma~\ref{lem:evalue-wadj}, the eigenvalues of $A$ are of the form
\begin{align}
	\xi_{\chi}(s,t) &= \frac{1}{\chi(id)} \sum_{i=1}^5 \omega_i(s,t) \chi(C_i)\label{eq:general_eigenvalues_odd}
\end{align}
for $\chi \in \operatorname{Irr}_n$.
The irreducible characters of degree less than $5\binom{n}{3}$ of $\sym(n)$ are given in Lemma~\ref{lem:irrecharless}. The eigenvalue that corresponds to $\chi^{[n]}$ is $\binom{n}{3}-1$ and the eigenvalue that corresponds to $\chi^{[n-1,1]},\ \chi^{[n-2,2]}$, and $\chi^{[n-3,3]}$ is $-1$. We use Table~\ref{table:odd} in Appendix~\ref{appendix} to compute the remaining eigenvalues corresponding to small degree characters. The eigenvalues that correspond to the other irreducible characters of degree less than $5\binom{n}{3}$ are
\begin{align}
	\begin{split} 
	\xi_{\chi^{[1^n]}} &= (-1)^{n-1} \left( -s -3t + \beta + 2 \gamma \right),\\
	\xi_{\chi^{[2,1^{n-2}]}} &= (-1)^n  + \frac{(-1)^{n-1}}{n-1} \left( s + t + \alpha -\beta - 2\gamma \right),\\
	\xi_{\chi^{[2^2,1^{n-4}]}} &= \frac{(-1)^n }{\binom{n}{2}-n} (\delta + s-t),\\
	\xi_{\chi^{[n-2,1^2]}} &= \frac{1}{\binom{n-1}{2}} \left( -s-t +\beta +\gamma \right),\\
	\xi_{\chi^{[3,1^{n-3}]}} &= \frac{(-1)^{n-1}}{\binom{n-1}{2}} \left( -s -t +\beta + \gamma \right),\\
	\xi_{\chi^{[2^3,1^{n-6}]}} &= (-1)^n + \frac{(-1)^{n-1}}{\binom{n}{3}-\binom{n}{2}} \left( s + t + \alpha - \beta -2\gamma \right),\\
	\xi_{\chi^{[n-3,1^3]}} &= \frac{1}{\binom{n-1}{3}} \left( s +t -(\beta +\gamma) \right),\\
	\xi_{\chi^{[4,1^{n-4}]}} &= \frac{(-1)^{n-1}}{\binom{n-1}{3}} \left( s + t -(\beta +\gamma) \right),\\
	\xi_{\chi^{[n-3,2,1]}} &= \frac{3}{n(n-2)(n-4)} (s+t),\\
	\xi_{\chi^{[3,2,1^{n-5}]}} &= \frac{3(-1)^{n-1}}{n(n-2)(n-4)} (s-t).
	\end{split} \label{eq:evalues-odd}
\end{align}
Let $\mathcal{P}$ be the polytope of $\mathbb{R}^2$ defined by the halfspaces
\begin{align}
\begin{split} 
	\left\{ 
			\begin{aligned}
				& 3x + y  < \beta +\gamma, \\
				&-\frac{n(n-2)(n-4)}{3} < y-x \leq \beta + \gamma - \binom{n-1}{3}, \\
				&\beta + \gamma - \binom{n-1}{3} < x + y < \beta + \gamma .
			\end{aligned}
	\right.
\end{split} 	\label{eq:polytope-odd}
\end{align}
The third relation of \eqref{eq:polytope-odd} is the set of all points between the parallel lines $(L_1):\ y = -x + \beta + \gamma - \binom{n-1}{3}$ and $(L_2) :\ y = -x + \beta +\gamma $. The second equation of \eqref{eq:polytope-odd} is the set of points between the parallel lines $(L_3):\ y=x -\frac{n(n-2)(n-4)}{3}$ and $(L_4):\ y = x + \beta + \gamma -\binom{n-1}{3}$. Further, $(L_1)$ and $(L_2)$ are both perpendicular to $(L_3)$ and $(L_4)$. The first equation in \eqref{eq:polytope-odd} is the halfspace below the line $(L_5) :\ y = -3x + \beta + \gamma$. The intersection $I$ of $(L_4)$ and $(L_5)$ has coordinate 
$$\left(\frac{1}{4}\binom{n-1}{3},\beta + \gamma  - \frac{3}{4} \binom{n-1}{3}\right).$$

\noindent We let the reader verify that $\mathcal{P}$ has four vertices. Moreover, the vertex $I$ has the maximum $y$-coordinate among the 4 vertices of $\mathcal{P}$, and the vertex of coordinate $\left( 0,\beta + \gamma - \binom{n-1}{3} \right)$ has the minimum $x$-coordinate. Therefore, $\mathcal{P} $  is non-empty and contained in the quadrant $ \left\{ (x,y) \in \mathbb{R}^2 : x>0,\ y<0 \right\}$.

\raggedbottom  It is easy to see that the eigenvalues in \eqref{eq:evalues-odd} are in the interval $\left(-1,\binom{n}{3}-1\right)$ for any $(t,s) \in \mathcal{P}$.
\subsection{Eigenvalues of large degree characters }\label{subs:large_char_odd}
We compute the eigenvalues of irreducible characters of degree larger than $5\binom{n}{3}$. 

For any $(t,s)\in \mathcal{P}$, we have the following
\begin{align*}
	\omega_1(s,t) &= -s -t +\beta +\gamma >0, \\
	\omega_2(s,t) &= \frac{1}{2} \left( -s -t + \alpha -\beta \right) > \frac{1}{2} \left( \alpha - 2\beta - \gamma - \binom{n-1}{2} \right) = \frac{1}{2} \left( \binom{n}{3} -2\binom{n}{2} \right) \\
	&>0 \ (\mbox{ for } n> 8),\\
	\omega_3(s,t) &= \frac{1}{2} s + \frac{1}{2} t + \frac{1}{2} ( \alpha + \beta -2\gamma)
	> \frac{1}{2} \left( \beta + \gamma - \binom{n-1}{3} +  \alpha - \beta - 2\gamma  \right) = 0.
\end{align*}
As $\mathcal{P} \subset \left\{ (x,y) \in \mathbb{R}^2 : x>0,\ y<0 \right\}$, we have $\omega_4(s,t) <0$ and $\omega_5(s,t)>0$. In conclusion, for $(t,s)\in \mathcal{P}$, the weights in \eqref{eq:solution_odd} are positive except $\omega_4(s,t)$.

By the general form of an eigenvalue of $A$ (see \eqref{eq:general_eigenvalues_odd}) and by Lemma~\ref{lem:main_val}, if $\chi \in \operatorname{Irr}_n$ with $\chi(id) > 5\binom{n}{3}$, then for any $(t,s) \in \mathcal{P}$
\begin{align*}
	|\xi_{\chi}(s,t)| &\leq \frac{1}{\chi(id)} \sum _{i=1}^5 |\omega_i(s,t)| | \chi(C_i)|\\
	&\leq  \frac{1}{5\binom{n}{3}} \left(\sum _{i\in\{1,2,3,5\}} \omega_i(s,t) - \omega_4(s,t)\right), \\ &=\frac{\binom{n}{3}-1-2s}{5\binom{n}{3}} \\
	&< \frac{\binom{n}{3}+\frac{2n(n-2)(n-4)}{3}}{5\binom{n}{3}} \ \ \ \ \ (\mbox{see the definition of }\mathcal{P}\ \eqref{eq:polytope-odd}) \\
	&= \frac{\binom{n}{3}\left(1+4\frac{n-4}{n-1}\right)}{5\binom{n}{3}}< \frac{5\binom{n}{3}}{5\binom{n}{3}}
	=1.
\end{align*}
We formulate this result as the following lemma.
\begin{lem}
	If $\chi \in \operatorname{Irr}_n$ with $\chi(id) > 5\binom{n}{3}$, then the eigenvalue of $A$ corresponding to $\chi$ is strictly less than $1$ in absolute value.
\end{lem}

The eigenvalues of the weighted adjacency matrix $A$ are therefore in the interval $\left[-1,\binom{n}{3}-1\right]$. Combining Subsection~\ref{subs:small_char_odd} and Subsection~\ref{subs:large_char_odd}, we conclude that the only irreducible characters of $\sym(n)$ giving the eigenvalue $-1$ are $\chi^{[n-1,1]},\ \chi^{[n-2,2]}$, and $\chi^{[n-3,3]}$. This completes the proof when $n\geq 27$ is odd.
\raggedbottom 
\section{$3$-setwise intersecting when $n\geq 20$ is even}\label{sect:proof-even}
A similar approach to the one in the previous section is used to prove Theorem~\ref{thm:main}, for $n\geq 20$ even. We consider the following weighted adjacency matrix 
\begin{align}
	A &= x_1 A_{(n-5,5)} + x_2 A_{(n-6,2^3)} + x_3 A_{(n-6,4,1^2)} + x_4 A_{(n-6,4,2)} + x_5 A_{(n-6,5,1)}.\label{eq:weighted_adjacency_matrix_even}
\end{align}
We find the weights $(x_i)_{i=1,2,3,4,5} \subset \mathbb{R}$ such that the eigenvalues of $A$ are in the interval $[-1,\binom{n}{3}-1]$. 
The values of the irreducible characters of $\sym(n)$ corresponding to $\chi^{[n]},\ \chi^{[n-1,1]},\ \chi^{[n-2,2]}$, and $\chi^{[n-3,3]}$ on the chosen conjugacy classes are given in the following table.

\begin{table}[H]
	\begin{longtable}{|c|c|c|c|c|c|} \hline
		& $C_{(n-5,5)}$  & $C_{(n-6,2^3)}$ & $C_{(n-6,4,1^2)}$ & $C_{(n-6,4,2)}$ & $C_{(n-6,5,1)}$ \\
		Representation  & & & & &\\ \hline
		$\chi^{[n]}$      &  $1 $ & $1$ & $1$  & $1$ &  $1$ \\ \hline
		$\chi^{[n-1,1]}$ & $-1$ & $-1$  & $1$ & $-1$ & $0$ \\ \hline
		$\chi^{[n-2,2]}$ &$0$ &  $3$ & $-1$ & $1$ & $-1$ \\ \hline
		$\chi^{[n-3,3]}$ &$0$ &  $-3$ & $-1$ & $-1$ & $0$ \\ \hline
		\caption{Values of the constituents of the permutation character on the chosen conjugacy classes.}
	\end{longtable}
\end{table}

We define $C_1 := C_{(n-5,5)},\ C_2 := C_{(n-6,2^3)},\ C_3:= C_{(n-6,4,1^2)},\ C_4 := C_{(n-6,4,2)}$, and $C_5 := C_{(n-6,5,1)}$. Let $\alpha = \binom{n}{3}-1$ and let $\beta,\gamma, \mbox{ and }\delta$ be, respectively, the degree of the irreducible characters $\chi^{[n-1,1]},\ \chi^{[n-2,2]}$, and $ \chi^{[n-3,3]}$. 
For $i\in \{1,2,3,4,5\}$, let $\omega_i = |C_i| x_i$. Similar to the previous section, we solve the system of linear equations 
\begin{align}
	\left\{
	\begin{aligned}
		\omega_1 + \omega_2 + \omega_3 + \omega_4 +\omega_5 &= \alpha,\\
		-\omega_1  - \omega_2 + \omega_3 - \omega_4 &= -\beta,\\
		3\omega_2 -\omega_3 + \omega_4 -\omega_5 &= -\delta,\\
		-3\omega_2 -\omega_3 -\omega_4 &= -\gamma.
	\end{aligned}
	\right.\label{eq:evalue-eq-even}
\end{align}
A solution to \eqref{eq:evalue-eq-even} is a function of two parameters, $t$ and $s$, and is of the form
\begin{align}
	\begin{split}
		\omega_1(s,t) &= -\frac{2}{3}t - \frac{2}{3}s  + \frac{1}{3}\alpha +  \frac{2}{3}\beta + \frac{\gamma}{3},\\
		\omega_2(s,t) &= \frac{1}{6}t - \frac{1}{3}s  + \frac{1}{6}(\alpha - \beta) - \frac{\gamma}{3},\\
		\omega_3(s,t) &= -\frac{1}{2}t + \frac{1}{2}(\alpha - \beta),\\
		\omega_4(s,t) &= s,\\
		\omega_5(s,t) &= t.\label{eq:weights-even}
	\end{split}
\end{align}
By Lemma~\ref{lem:evalue-wadj}, the eigenvalue of $A$ corresponding to the irreducible character $\chi \in \operatorname{Irr}_n$ is
\begin{align}
	\xi_{\chi} (s,t) &= \frac{1}{\chi(id)}\sum_{i=1}^5 \omega_i(s,t) \chi(C_i).\label{eq:general-eig-even}
\end{align}

\subsection{Irreducible characters  on $C_1,C_2,C_3,C_4$, and $C_5$}\raggedbottom
In this subsection, we give the values of the irreducible characters of $\sym(n)$ on the conjugacy classes $C_1,\ C_2,\ C_3,\ C_4$, and $C_5$. Contrary to the conjugacy classes in Section~\ref{sect:odd}, the values of the irreducible characters on the chosen conjugacy classes are in the set $\{0,\pm 1,\pm 2,\pm 3\}$.
\begin{lem}
	If $\chi \in \operatorname{Irr}_n$, then $|\chi(x)| \in \{0,1,2,3\}$ for any $x\in C_1 \cup C_2\cup C_3\cup C_4\cup C_5$.\label{lem:main_val-even}
\end{lem}
The proof is similar to the proof of Lemma~\ref{lem:main_val}.

\begin{proof}
	Let $\lambda \vdash n$. By Lemma~\ref{lem:uniqueness} and the Murnaghan-Nakayama rule, we have
	\begin{align*}
		|\chi^\lambda_{(n-5,5)}| & \leq \max_{\mu \vdash 5} |\chi^\mu_{(5)}| = 1,\\
		|\chi^\lambda_{(n-5,2^3)}| & \leq \max_{\mu \vdash 6 } |\chi^\mu_{(2^3)}| = 3,\\
		|\chi^\lambda_{(n-5,4,1^2)}| & \leq \max_{\mu \vdash 6} |\chi^\mu_{(4,1^2)}| = 1,\\
		|\chi^\lambda_{(n-6,4,2)}| & \leq \max_{\mu \vdash 6} |\chi^\mu_{(4,2)}| = 1,\\
		|\chi^\lambda_{(n-6,5,1)}| & \leq \max_{\mu \vdash 6} |\chi^\mu_{(5,1)}| = 1.
	\end{align*}
	This completes the proof.
\end{proof}
\subsection{Eigenvalues of small degree characters}\label{subsect:even-small-degree}
In this subsection, we compute the eigenvalues coming from irreducible characters of degree less than $3\binom{n}{3}$.

The eigenvalues corresponding to the irreducible constituents of the permutation character are $\xi_{\chi^{[n]}} = \binom{n}{3}-1$ and $\xi_{\chi^{[n-1,1]}} = \xi_{\chi^{[n-2,2]}} = \xi_{\chi^{[n-3,3]}} = -1$. By \eqref{eq:general-eig-even}, Table~\ref{table:even} Appendix~\ref{appendix}, and Lemma~\ref{lem:irrecharless}, the eigenvalues of degree less than $3\binom{n}{3}$ are  

\begin{align}
\begin{split} 
\xi_{\chi^{[1^n]}} &= (-1)^n \left( \binom{n}{3}-1 -2s -2t \right),\\
\xi_{\chi^{[n-2,1^2]}} &= \frac{1}{\binom{n-1}{2}} \left( -t + \beta +\gamma \right),\\
\xi_{\chi^{[3,1^{n-3}]}} &= \frac{(-1)^{n}}{\binom{n-1}{2}} \left( -t +\beta + \gamma \right) ,\\
\xi_{\chi^{[n-3,1^3]}} &= \frac{1}{\binom{n-1}{3}} \left( t - (\beta + \gamma) \right),\\
\xi_{\chi^{[4,1^{n-4}]}} &= \frac{(-1)^n}{\binom{n-1}{3}} \left(t - (\beta +\gamma)\right) ,\\
\xi_{\chi^{[n-3,2,1]}} &= \frac{3t}{n(n-2)(n-4)} ,\\
\xi_{\chi^{[3,2,1^{n-5}]}} &= \frac{(-1)^{n-1}3t}{n(n-2)(n-4)},\\
\xi_{\chi^{[2,1^{n-2}]}}&= (-1)^{n-1} +\frac{2s}{n-1},\\
\xi_{\chi^{[2^2,1^{n-4}]}} &=(-1)^{n-1} + \frac{(-1)^n}{\binom{n}{2} - n} \left( -2s +2t \right),\\
\xi_{\chi^{[2^3,1^{n-6}]}} &= (-1)^{n-1} + \frac{(-1)^n2s}{\binom{n}{3} - \binom{n}{2}}.
\end{split} 
\label{eq:eig-even}
\end{align}

Let $\mathcal{P}$ be the polytope defined by the halfspaces
\begin{align}
	\left\{
		\begin{aligned}
			&2x + 2y < \binom{n}{3},\\
			&x - y >0,\\
			& \beta + \gamma - \binom{n-1}{2} < x < \beta + \gamma + \binom{n-1}{2},\\
			&y>0.
		\end{aligned}
	\right.\label{eq:polytope}
\end{align}
The first two equations determine two lines intersecting at a point $I\in$$\left\{(x,y) \in \mathbb{R}^2 : x,y\geq 0 \right\}$. The coordinates of this intersection point, $I$, is $\left( \frac{1}{4}\binom{n}{3},\frac{1}{4}\binom{n}{3} \right)$. It is easy to check that $\frac{1}{4}\binom{n}{3} > \beta + \gamma - \binom{n-1}{2}$. Hence, we conclude that the polytope $\mathcal{P} $ is non-empty and is contained in the quadrant $ \left\{(x,y) \in \mathbb{R}^2 : x,y> 0 \right\}$. 

We let the reader verify that all the eigenvalues corresponding to irreducible characters of degree less than $3\binom{n}{3}$ (i.e., the eigenvalues in \eqref{eq:eig-even}) are in $\left(-1,\binom{n}{3}-1\right)$, for any weights $\omega_i(s,t)$, $1\leq i \leq 5$, and $(t,s) \in \mathcal{P}$.
\subsection{Eigenvalues of large degree characters}\label{subsect:even-large-degree}
We prove that the irreducible characters of degree greater than $ 3\binom{n}{3}$ are strictly less than $1$ in absolute value. 

We prove that $\omega_i(s,t) \geq 0$, for any $i\in \{1,2,3,4,5\}$ and $(t,s) \in \mathcal{P}$. It is straightforward to see that $\omega_4(s,t)$ and $\omega_5(s,t)$ are non-negative whenever $s,t\geq 0$.  Using the relations in \eqref{eq:polytope}, the other weights are positive since 
\begin{align*}
	\omega_1(s,t) & \geq \frac{2}{3}\beta + \frac{1}{3}\gamma - \frac{1}{3}>0,\\
	\omega_2(s,t) & \geq \frac{1}{6} \left( \alpha - 3\gamma -2\beta -\binom{n-1}{2} \right) >0,\\
	\omega_3 (s,t) &\geq \frac{1}{4} \left( \alpha -2\beta -1  \right) + \frac{1}{2}s >0,
\end{align*}
when $n\geq 16$. \\
In conclusion, the weights in \eqref{eq:weights-even} are all positive for any $(t,s) \in \mathcal{P}$.

Now, we are equipped with all the tools for the remaining part of the proof. If $\chi \in\operatorname{Irr}_n$ is of degree greater than $3\binom{n}{3}$ and $(t,s) \in \mathcal{P}$, then by Lemma~\ref{lem:main_val-even}, we have
\begin{align*}
	|\xi_{\chi}(s,t)| &\leq \frac{1}{\chi(id)} \sum_{i=1}^5 |\omega_i(s,t)||\chi(C_i)|\\
	&\leq \frac{1}{\chi(id)}  \sum_{i=1}^5 3|w_i(s,t)|\\
	&\leq \frac{3}{3\binom{n}{3}}\sum_{i=1}^5 \omega_i(s,t) = \frac{\binom{n}{3}-1}{\binom{n}{3}} <1.
\end{align*}

With the results in Subsection~\ref{subsect:even-small-degree}, we conclude that all eigenvalues of the weighted adjacency matrix $A$ (defined in \eqref{eq:weighted_adjacency_matrix_even}) are in the interval $[-1,\binom{n}{3}-1]$. Moreover, the only irreducible characters of $\sym(n)$ giving the eigenvalue $-1$ are $\chi^{[n-1,1]},\ \chi^{[n-2,2]}$, and $\chi^{[n-3,3]}$. This completes the proof for the case $n\geq 20$ even.

\section{Further works}\label{sect:further_works}
In this paper, we proved the first part of Conjecture~\ref{conj:setwise}, for $t=3$. The case $t=2$ was recently proved in \cite{meagher20202}. We believe that the method used in \cite{meagher20202} and the one used in this paper will not work for the case $t\geq 4$. This is mainly due to the fact that more conjugacy classes of $t$-derangements that are not just long cycles will be used. As we saw in the case of the conjugacy class with cycle type $(n-6,2^3)$, the values of the irreducible characters on such conjugacy classes might be relatively small, in absolute value, but it forces us to examine many eigenvalues that are not bounded by $1$, in absolute value.

Ellis, Friedgut, and Pilpel \cite{ellis2011intersecting} proved that for any fixed $t\in \mathbb{N}$ and $n\in \mathbb{N}$ large enough depending on $t$, if any two permutations of $\mathcal{F} \subset \sym(n)$ agree on at least $t$ elements, then $|\mathcal{F}|\leq (n-t)!$. Ellis \cite{ellis2012setwise} proved the analogue of this result for $t$-setwise intersecting permutations. In particular, the result in \cite{ellis2012setwise} implies that Theorem~\ref{thm:main} holds for $n$ large enough; however, we could not find the exact lower bound on $n$.

Although we did not characterize the maximum cocliques of $\Gamma_{n,3}$, our method gives a complete proof for the case $t = 3$ of the first part of Conjecture~\ref{conj:setwise}. When $n\geq 11$, we were also able to prove that there are infinitely many adjacency matrices that would work to prove Theorem~\ref{thm:main}. Another strong aspect of our result is that only ten conjugacy classes (five for each case) were used, in contrast to the result in \cite{ellis2012setwise}, where all the conjugacy classes of $t$-derangements are used. 

We end this paper by asking a few questions for future works.
\begin{quest} 
	Is there a profound reason behind why the conjugacy classes that we chose to construct the weighted adjacency matrices work, or is it merely because they are mostly long cycles?
\end{quest}

Godsil and Meagher \cite{godsil2009new} used a rank argument on a matrix whose columns are characteristic vectors of maximum cocliques to characterize the maximum $1$-setwise intersecting families of $\sym(n)$. The first step in their proof is to show that the characteristic vectors of every maximum coclique lie in the eigenspace $U_{[n]} \oplus U_{[n-1,1]}$, which is the eigenspace induced by the permutation character for the natural action of $\sym(n)$. Similar results are also obtained in this paper for $3$-setwise intersecting permutations and in \cite{meagher20202} for $2$-setwise intersecting permutations. We ask the following.
\begin{quest}
	Is it possible to characterize the maximum $2$-setwise and $3$-setwise intersecting families of permutations of $\sym(n)$ using the rank argument in \cite{godsil2009new}?
\end{quest}

Another interesting direction is the EKR property for transitive subgroups of $\sym(n)$; i.e., transitive permutation groups of degree $n$. It follows easily from the ``No-Homomorphism Lemma'' \cite{albertson1985homomorphisms} that if $H \mbox{ and } G$ are transitive permutation groups of degree $n$ such that $H \leq G$ and $H$ has the EKR property, then $G$ has the EKR property.
\begin{prob}
	Let $n\geq 2$. Find the maximum $k \in \mathbb{N}$ so that the chain $\sym(n)= G_1 > G_2> \cdots > G_k $ of transitive subgroups of $\sym(n)$ has the following properties:
	\begin{enumerate}[(1)]
		\item no proper transitive subgroups of $G_k$ has the EKR property, and
		\item $G_i$ has the EKR property for any $i\in \{1,2,\ldots,k\}$.
	\end{enumerate}
\end{prob}

An immediate consequence of Theorem~\ref{thm:main} is that $\chi(\Gamma_{n,3}) = \binom{n}{3}$. Indeed, one can always assign distinct colors to the cosets of a setwise stabilizer of a $3$-subset of $[n]$. There are $\binom{n}{3}$ such cosets and each of them is a coclique of $\Gamma_{n,3}$. Thus, $\chi(\Gamma_{n,3}) \leq \binom{n}{3}$. On the other hand, Theorem~\ref{thm:main} implies that $\chi(\Gamma_{n,3}) \geq \frac{n!}{\alpha(\Gamma_{n,3})} = \binom{n}{3}$, for $n\geq 11$. We conjecture the following.
\begin{conj}
	For any positive integers $n$ and $t$ such that $n\geq t$, we have $\chi(\Gamma_{n,t}) = \binom{n}{t}$.
\end{conj}

{\noindent\textsc{Acknowledgment}}: We would like to thank Karen Meagher and Shaun Fallat for reading an earlier draft of this paper, and helping us improve the presentation of the paper. We are also grateful to the two anonymous referees for carefully reading the manuscript and for their valuable comments.


\appendix
\newgeometry{margin=1in,top=1in}
\section{Character values}\label{appendix}

\begin{longtable}{| c | c | c | c | c | c | c |}
	\hline
	\textbf{Representation} & \textbf{Degree} & $C_{(n-2,1^2)}$ & $C_{(n-2,2)}$ & $C_{(n)}$ & $C_{(n-5,4,1)}$ & $C_{(n-1,1)}$ \endhead
	\hline
	$\chi^{[1^n]}$ & $1$ &$(-1)^{n-1}$ &$(-1)^n$ & $(-1)^{n-1}$& $(-1)^{n-1}$ & $(-1)^n$ \\
	\hline
	$\chi^{[2,1^{n-2}]}$ & $n-1$ &$(-1)^{n-1}$ &$(-1)^{n-1}$ &$(-1)^n$ & $0$ & $0$ \\
	\hline
	$\chi^{[2^2,1^{n-4}]}$ & $\binom{n}{2} - n$ &$(-1)^n$ &$(-1)^n$ &$0$ & $(-1)^{n}$ &$(-1)^{n-1}$ \\
	\hline
	$\chi^{[2^3,1^{n-6}]}$ & $\binom{n}{3} - \binom{n}{2}$& $(-1)^n$ &$(-1)^{n-1}$ &$0$ &$0$ & $0$ \\
	\hline
	$\chi^{[3,1^{n-3}]}$ & $ \binom{n-1}{2}$& $0$&$0$ &$(-1)^{n-1}$ &$0$& $0$ \\
	\hline
	$\chi^{[3,2,1^{n-5}]}$ & $\frac{n(n-2)(n-4)}{3}$ &$0$ &$0$ &$0$ & $(-1)^{n-1}$& $0$ \\
	\hline
	$\chi^{[4,1^{n-4}]}$ & $\binom{n-1}{3}$  &$0$ &$0$ &$(-1)^n$ & $0$& $0$\\
	\hline
	$\chi^{[n-3,1^3]}$ &  $\binom{n-1}{3}$  &$0$ & $0$&$-1$ &$0$& $0$ \\
	\hline
	$\chi^{[n-3,2,1]}$ & $\frac{n(n-2)(n-4)}{3}$ &$0$ &$0$ &$0$ &$1$ & $0$\\
	\hline
	$\chi^{[n-3,3]}$ & $\binom{n}{3} - \binom{n}{2}$&$-1$ &$-1$ &$0$ &$0$& $0$ \\
	\hline
	$\chi^{[n-2,1^2]}$ & $ \binom{n-1}{2}$&$0$ &$0$ &$1$ & $0$& $0$\\
	\hline
	$\chi^{[n-2,2]}$ & $\binom{n}{2} - n$ &$-1$ &$1$ &$0$ &$-1$&$-1$ \\
	\hline
	$\chi^{[n-1,1]}$ & $n-1$ &$1$ &$-1$ &$-1$ &$0$& $0$ \\
	\hline
	$\chi^{[n]}$ & $1$ &$1$ &$1$ &$1$ &$1$ &$1$\\
	\hline
	\caption{Character values of irreducible characters of degree less than $5\binom{n}{3}$.}
	\label{table:odd}
\end{longtable}

\begin{longtable}{| c | c | c | c | c | c | c |c|}
	\hline
	\textbf{Representation} & \textbf{Degree} & $C_{(n-6,2^3)}$ & $C_{(n-6,5,1)}$ & $C_{(n-6,4,2)}$ & $C_{(n-6,4,1^2)}$ & $C_{(n-5,5)}$  \endhead
	\hline
	$\chi^{[1^n]}$ & $1$ & $(-1)^n$ & $(-1)^{n-1}$ & $(-1)^{n-1}$ & $(-1)^n$ & $(-1)^n$\\
	\hline
	$\chi^{[2,1^{n-2}]}$ & $n-1$ &$(-1)^{n-1}$  & $0$ &$(-1)^n$  &$(-1)^n $ & $(-1)^{n-1}$\\
	\hline
	$\chi^{[2^2,1^{n-4}]}$ & $\binom{n}{2}-n$& $3(-1)^n$ & $(-1)^n$ & $(-1)^{n-1}$ & $(-1)^{n-1}$ & $0$\\
	\hline
	$\chi^{[2^3,1^{n-6}]}$ &$\binom{n}{3} - \binom{n}{2}$ & $3(-1)^{n-1}$ & $0$ & $(-1)^n$ & $(-1)^{n-1}$ & $0$\\
	\hline
	$\chi^{[3,1^{n-3}]}$ & $\binom{n-1}{2}$& $2(-1)^{n-1}$ & $0$ & $0$ & $0$ & $(-1)^n$ \\
	\hline
	$\chi^{[3,2,1^{n-5}]}$ & $\frac{n(n-2)(n-4)}{3}$ & $0$  & $(-1)^{n-1}$ & $0$ & $0$ & $0$ \\
	\hline
	$\chi^{[4,1^{n-4}]}$ & $\binom{n-1}{3}$ & $2(-1)^n$ & $0$ & $0$ & $0$ &$(-1)^{n-1}$\\
	\hline
	$\chi^{[n-3,1^3]}$ & $\binom{n-1}{3}$ & $2$ & $0$ & $0$ & $0$ & $-1$\\
	\hline
	$\chi^{[n-3,2,1]}$ & $\frac{n(n-2)(n-4)}{3}$ & $0$ & $1$ & $0$ & $0$ &$0$ \\
	\hline
	$\chi^{[n-3,3]}$ & $\binom{n}{3}-\binom{n}{2}$& $-3$ & $0$ & $-1$ & $-1$ &$0$\\
	\hline
	$\chi^{[n-2,1^2]}$ & $\binom{n-1}{2}$ & $-2$ & $0$ & $0$ & $0$ &$1$\\
	\hline
	$\chi^{[n-2,2]}$ &$\binom{n}{2} - n$ & $3$ & $-1$ & $1$ & $-1$ &$0$\\
	\hline
	$\chi^{[n-1,1]}$ & $n-1$ & $-1$ & $0$ & $-1$ & $1$& $-1$\\
	\hline
	$\chi^{[n]}$& $1$ &$1$ & $1$ &$1$ &$1$ &$1$\\
	\hline
	\caption{Character values of the irreducible characters of degree less than $3\binom{n}{3}$.}
	\label{table:even}
\end{longtable}
\end{document}